\documentclass[11pt]{amsart}
\usepackage{mathrsfs}
\usepackage{amssymb}
\usepackage[utf8]{inputenc} 
\usepackage[english]{babel}

\theoremstyle{plain} 
\newtheorem{thm}{Theorem} 
\newtheorem{lem}[thm]{Lemma} 
 
\newtheorem{prop}[thm]{Proposition}

\theoremstyle{definition}

\theoremstyle{remark} 
\newtheorem*{rem}{Remark}

\DeclareMathOperator{\ess}{ess} 
\DeclareMathOperator{\ea}{ea}
\DeclareMathOperator{\Arg}{arg}
\DeclareMathOperator{\pv}{p.v. }
\DeclareMathOperator{\Tr}{Tr}
\DeclareMathOperator{\supp}{supp}

\title[Essential spectrum of the Neumann--Poincar\'e operator]{The essential spectrum of the Neumann--Poincar\'e operator on a domain with corners}
\date{\today} 

\author{Karl-Mikael Perfekt} \address{Department of Mathematical Sciences, Norwegian University of Science and Technology (NTNU), NO-7491 Trondheim, Norway} \email{karl-mikael.perfekt@math.ntnu.no}

\author{Mihai Putinar} \address{Mathematics Department, University of California, Santa Barbara, Ca 93106, \ {\rm and} School of Mathematics \& Statistics, Newcastle University Newcastle upon Tyne, NE1 7RU, United Kingdom} \email{mputinar@math.ucsb.edu, mihai.putinar@ncl.ac.uk}

\begin{document}
\begin{abstract}
	Exploiting the homogeneous structure of a wedge in the complex plane, we compute the spectrum of the anti-linear Ahlfors--Beurling transform acting on the associated Bergman space. 
	Consequently, the similarity equivalence between the Ahlfors--Beurling transform and the Neumann--Poincar\'e operator provides the spectrum of the latter integral operator on a wedge. A localization
	technique and conformal mapping lead to the first complete description of the essential spectrum of the Neumann--Poincar\'e operator on a planar domain with corners, with respect to the energy norm of the associated harmonic field.
\end{abstract} 

\maketitle
\section{Introduction}

Exactly a hundred years ago Torsten Carleman defended his doctoral dissertation titled "\"Uber das Neumann--Poincar\'esche Problem f\"ur ein Gebiet mit Ecken" \cite{Carleman16}.
The double-layer potential singular integral operator associated with a domain $\Omega \subset \mathbb{R}^2$, known also as the Neumann--Poincar\'e (NP) operator, was at that time a central object of study, first for its role in solving boundary value problems of mathematical physics, but also as the main example in the emerging abstract spectral theories proposed by Hilbert, Fredholm and F. Riesz. While the NP operator is compact on smooth boundaries, the presence of corners produces continua in its essential spectrum. For the modern reader these concepts make no sense without a well defined, complete functional space where the operator NP acts, not to mention also the current definitions of essential spectrum, spectral resolution, approximate or generalized eigenvalues, etc. In a tour de force Carleman did solve the singular integral equation governed by the NP operator and analyzed the (asymptotic) structure of its solutions in a domain with corners. He made use of elementary and very ingenious geometric transformations together with the, at his time new, theory of Fredholm determinants combined with the canonical factorization of entire functions of Hadamard. Carleman's work did not attract the visibility it deserves, nor
did the prior results of his predecessors, among which we mention Zaremba \cite{Zaremba04}.
%We are confident that Carleman's baroque computations hide deep observations, and their time to bloom is
%approaching fast.

Only a few years after Carleman's defense, Radon \cite{Radon} developed the theory of measures of bounded variation, and applied it to study the NP operator on the space of continuous functions $C(\partial \Omega)$. He computed the essential spectral radius for boundaries $\partial \Omega$ of bounded rotation, extending Carleman's work. Note that we now understand that for non-smooth boundaries, the spectrum of the NP operator depends drastically on the underlying space.  For instance, when the NP operator is considered on $L^p(\partial \Omega)$, $p \geq 2$, $\Omega$ a curvilinear polygon, the complete spectral picture is available \cite{Mitrea02} -- and it is entirely different from what appears in the work of Carleman and Radon. 

The Hilbert space on which we will perform a spectral analysis is the energy space of potential fields with sources carried by $ \partial \Omega$. The energy space was advocated by Poincar\'e in his foundational and novel approach to the Dirichlet problem. It stands out as a natural setting for the NP operator for at least two reasons. First, the invertibility properties of the NP operator acting on the energy space lead to finite energy solutions of boundary value problems for the Laplacian, and such solutions often carry a physical interpretation.  Second, due to a symmetrization property, the NP operator has real spectrum on the energy space even when the boundary $\partial \Omega$ is non-smooth (this is not true for example on $L^2(\partial \Omega)$). The recent survey \cite{Wendland09} treats among other things qualitative aspects of the essential spectrum of the NP operator on various spaces of interest, for domains with corners. However, the case of the physically motivated energy space is noted for the lack of information concerning the structure of the essential spectrum.

Ahlfors \cite{Ahlfors52} observed a connection between the spectral radius of the NP operator (the largest Fredholm eigenvalue of $\Omega$) and the quasiconformal reflection coefficient of $\partial \Omega$. The reflection coefficient is notoriously difficult to compute for general domains which do not have any special geometric structure \cite{Krushkal09}. Yet, Ahlfors' inequality provides to date nearly all known spectral bounds of the NP operator in the energy norm. 

Very recently, the NP operator has received a resurgence of interest arising from the mathematical theory of new materials and its need to solve various inverse problems. In particular, spectral analysis questions of the NP operator on non-smooth domains are central in the penetrating works of Ammari, Kang, Milton and their enthusiastic collaborators \cite{ACKLM13, ACKLY13, ADKL14, AK04}. And again, a century after Carleman's work, estimates of the location of the essential spectrum of the NP operator and the asymptotic behavior of its generalized eigenfunctions turn out to be highly sought results. The preprint \cite{KLY15} contains a detailed description of the spectral resolution of the NP operator acting on a lens domain, with respect to the energy space. In a previous work \cite{PP14} we have obtained bounds for the spectrum of the NP operator, in the same energy space, on domains with corners, via distortion theorems of conformal mappings. In the same setting, a detailed numerical study of the spectrum has been done in the preprint \cite{HKL16}.  Some interesting geometric analysis questions pertaining to the NP operator also appear in \cite{MS15}.

The present note contains a new approach to the spectral analysis of the NP operator in a wedge in two variables. We exploit the similarity between the NP operator acting on the energy space (identifiable with a fractional Sobolev space on the boundary) and the Ahlfors--Beurling singular integral operator \cite{BS51} acting on the Bergman space of the underlying domain. The homothetic action of the commutative group of positive real numbers on the wedge turns out to simplify, at least conceptually, the computation. We then generalize, via a standard localization procedure and a conformal mapping technique, the wedge computation to domains
with finitely many corners. The outcome is an exact picture of the essential spectrum of the NP operator, in the energy norm. 

A few comments on the nature of the singular integral transformations we deal with are in order: the NP operator is not symmetric, but only symmetrizable in a norm which is equivalent to a fractional
Sobolev space norm on the boundary, see \cite{PP14} for details. Therefore the NP operator is only scalar in the sense of Dunford, and any spectral resolution has to be understood in this generalized sense \cite{DS71}. The spectral analysis of the anti-linear Ahlfors--Beurling operator can today be naturally understood within the abstract theory of complex symmetric operators \cite{GPP14}. While our localization and conformal mapping argument shows that the essential spectrum is a continuum, it does not control the possible singular continuous part of the spectrum. Recall that the absolute continuity of the spectrum of an operator can be altered by a Hilbert-Schmidt perturbation (according to the classical Weyl-von Neumann theorem). Conversely, for a self-adjoint operator it is preserved by a trace class perturbation (according to the Kato-Rosenblum theorem). Let us therefore clarify that our localizations, while Dunford scalar, can not be jointly put on a normal form. On the other hand, it is known \cite{Werner97} that as a rectangle is elongated, the spectral radius of the associated Ahlfors--Beurling operator changes. When combined with the results on the essential spectrum of the present article, it follows that isolated eigenvalues of the NP operator depend on the non-local geometry of the domain. The eigenvalues are thus very unlikely to be given a simple description.

\section{Preliminaries}
Let $\Omega \subset \mathbb{C}$ be a bounded Lipschitz domain with connected boundary. The Sobolev space of order $1/2$ along the boundary, $H^{1/2}(\partial \Omega)$, is defined in the usual way, using a bi-Lipschitz atlas to view $\partial \Omega$ as a manifold. A Hilbert space norm on $H^{1/2}(\partial \Omega)$ is given by the Besov norm
\begin{equation} \label{eq:besnorm}
\| f \|_{H^{1/2}(\partial \Omega)}^2 \sim \| f \|_{L^2(\partial \Omega)}^2 + \int_{\partial \Omega \times \partial \Omega} \frac{|f(x)-f(y)|^2}{|x-y|^{2}} \, d \sigma(x) \, d \sigma(y), 
\end{equation} 
where $\sigma$ is the natural Hausdorff measure on $\partial \Omega$. See for instance \cite[Appendix II]{TW09}. $H^{-1/2}(\partial \Omega)$ is then defined by duality with respect to the pairing of $L^2(\partial \Omega)$, and $H_0^{-1/2}(\partial \Omega)$ is its subspace of elements $f \in H^{-1/2}(\partial \Omega)$ such that $\langle f, 1 \rangle_{L^2(\partial \Omega)} = 0$.

The Neumann--Poincare operator $K \colon H^{1/2} \left( \partial \Omega \right) \to H^{1/2} \left( \partial \Omega \right)$ is defined by
\begin{equation} \label{eq:NPopdef}
Kf(x) = \frac{2}{\pi} \pv\int_{\partial \Omega}  \partial_{n_y} \log|x-y| f(y) \, d\sigma(y), \qquad x \in \partial \Omega,
\end{equation}
where $n_y$ is the outward normal derivative of $\partial \Omega$ at $y$. $K$ is always a bounded operator. When evaluating the integral for $x \notin \partial \Omega$, we obtain the harmonic double-layer potential $Df$,
\[
Df(x) = \frac{1}{\pi}\int_{\partial \Omega}  \partial_{n_y} \log|x-y| f(y) \, d\sigma(y), \qquad x \in \mathbb{C} \setminus \partial \Omega = \Omega \cup \overline{\Omega}^c.
\]
Yet another characterization of $H^{1/2}(\partial \Omega)$ is that it consists precisely of the functions $f$ such that $Df \in H^1(\mathbb{C} \setminus \partial \Omega)$, i.e. such that
\[\|Df\|_{H^1(\mathbb{C} \setminus \partial \Omega)} = \int_{\Omega \cup \overline{\Omega}^c} | \nabla Df |^2 \, dx < \infty. \]
In other words, $H^{1/2}(\partial \Omega)$ is the space of charges which yield potentials of finite energy. As an element of $H^1(\Omega)$, $Df$ has an (interior) trace $\Tr Df \in H^{1/2}(\partial \Omega)$. $Df$ and $Kf$ are related by the jump formula \begin{equation} \label{eq:jumpformula}
\Tr Df = \frac{1}{2}(f + Kf).
\end{equation}
In the case that $\partial \Omega$ is a $C^2$-curve, \cite[Ch. 8]{Kress} offers a very readable and self-contained introduction to the Neumann--Poincar\'e operator and its use in constructing finite energy solutions to the Dirichlet and Neumann problems for the Laplacian. 

 By $K^* : H^{-1/2}(\partial \Omega) \to H^{-1/2}(\partial \Omega)$ we mean the adjoint of $K$ with respect to the  $L^2 \left( \partial \Omega \right)$-pairing. In \cite{PP14} the authors showed, for a general Lipschitz domain, that $K^* : H_0^{-1/2}(\partial \Omega) \to H_0^{-1/2}(\partial \Omega)$ is similar to a self adjoint operator. The only effect of considering the action of $K^*$ on $H_0^{-1/2}(\partial \Omega)$ rather than on $H^{-1/2}(\partial \Omega)$ is that it loses its isolated eigenvalue $\lambda = 1$ of multiplicity $1$. 

One realization of such a self adjoint operator is the anti-linear Ahlfors--Beurling operator $T_\Omega$ acting on the Bergman space $L^2_a(\Omega)$,
\begin{equation} \label{eq:beurlingtransform}
T_\Omega f (z) = \frac{1}{\pi} \pv \int_\Omega \frac{\overline{f(\zeta)}}{(\zeta-z)^2} \, dA(\zeta), \qquad f\in L^2_a(\Omega), \, z \in \Omega.
\end{equation}
Here the Bergman space $L^2_a(\Omega)$ consists of the holomorphic square-integrable functions in $\Omega$.
To be precise, $K^\ast \colon H_0^{-1/2} \left( \partial \Omega \right) \to H_0^{-1/2} \left( \partial \Omega \right)$ and $T_\Omega : L^2_a(\Omega) \to L^2_a(\Omega)$ are similar as $\mathbb{R}$-linear operators. 

From the similarity we have the following equality of spectra: 
\[\sigma(K) = \sigma_\mathbb{R}(T_\Omega) \cup \{1\}.\]
 Note that if $\lambda$ is in the spectrum of $T_\Omega$, then by anti-linearity so is $e^{i\theta} \lambda$ for any $\theta \in \mathbb{R}$. However, we are interested only in the real spectrum, $K : H^{1/2}(\partial \Omega)/\mathbb{C} \to H^{1/2}(\partial \Omega)/\mathbb{C}$ being similar to a self-adjoint operator over the complex field. Therefore, to determine the spectrum of $K$ using $T_\Omega$ we consider only $\lambda \geq 0$ in the spectrum of $T_\Omega$, and note that the spectrum of $K$ consists of $\pm \lambda$, for all such points $\lambda$, in addition to the simple eigenvalue $1$. 

The Neumann--Poincar\'e operator \eqref{eq:NPopdef} may be written more explicitly as
\[
Kf(x) = \frac{2}{\pi} \pv\int_{\partial \Omega}  \frac{\langle y - x, n_y\rangle}{|y-x|^2} f(y) \, d\sigma(y), \qquad x \in \partial \Omega,
\]

In this article we shall consider the case where $\Omega \subset \mathbb{C}$ is a $C^{2}$-smooth curvilinear polygon. By this we mean that $\Omega$ is a bounded and simply connected domain whose boundary is curvilinear polygonal: there are a finite number of counter-clockwise consecutive vertices $(a_j)_{j=1}^N \subset \mathbb{C}$, $1 \leq N < \infty$, connected by $C^{2}$-smooth arcs $\gamma_j : [0,1] \to \mathbb{C}$ with starting point $a_j$ and end point $a_{j+1}$ (indices modulo $N$), such that $\partial \Omega = \cup_{j} \gamma_j$ and $\gamma_j$ and $\gamma_{j+1}$ meet at the interior angle  $\alpha_{j+1}$ at $a_{j+1}$, $0 < \alpha_{j+1} < 2\pi$. Note that if $\gamma \subset \partial \Omega$ is a $C^2$ subarc, then the kernel 
\begin{equation}
k(x,y) = \frac{\langle y - x, n_y\rangle}{|y-x|^2}
\end{equation}
is actually bounded and continuous for $y \in \gamma'$, $x \in \partial \Omega$, where $\gamma'$ is any strict subarc of $\gamma$. Similarly, $k(x,y)$ is bounded and continuous for $x \in \gamma_1$ and $y \in \gamma_2$ if $\gamma_1$ and $\gamma_2$ are compact and disjoint subsets of $\partial \Omega$ (regardless of smoothness). From these observations we obtain the compactness of certain cut-offs which we shall later use to localize the operator $K$. We roughly follow the proof from \cite{Kress} that $K$ is compact if $\partial \Omega$ is $C^2$, in addition to a trivial observation about multiplication operators.
\begin{lem} \label{lem:multbound}
Let $\rho$ be a smooth function on $\partial \Omega$. Then $M_\rho$, the operator of multiplication by $\rho$, is a bounded operator on $H^{1/2}(\partial \Omega)$.
\end{lem}
\begin{proof}
This is a straightforward consequence of the Besov norm expression \eqref{eq:besnorm} for $H^{1/2}(\partial \Omega)$.
\end{proof}
\begin{lem} \label{lem:cutoffcpct}
Let $\Omega$ be a $C^{2}$-smooth curvilinear polygon with vertices $(a_j)_{j=1}^N$. For each $j$, $1 \leq j\leq N$, let $\rho_j$ be a smooth function on $\partial \Omega$ such that $\rho_j(x) = 1$ for all $x$ in a neighborhood of $a_j$. Furthermore, suppose that the supports of $\rho_j$ are pairwise disjoint (at a positive distance apart). Let $\rho_{N+1} = 1 - \sum_{j=1}^N \rho_j$.

If $j \neq k$ or $j = k = N+1$, then $M_{\rho_j} K M_{\rho_k} :    H^{1/2}(\partial \Omega) \to H^{1/2}(\partial \Omega)$ is a compact operator.
\end{lem}
\begin{proof}
Let $\frac{d}{d\sigma}$ denote (tangential) differentiation along $\partial \Omega$, extended in the distributional sense to an operator $\frac{d}{d\sigma} : H^{1/2}(\partial \Omega) \to H^{-1/2}(\partial \Omega)$. We will actually show that $M_{\rho_j} K M_{\rho_k}$ is bounded as a map from $H^{1/2}(\partial \Omega)$ into $H^1(\partial \Omega)$, where $H^1(\partial \Omega)$ is the space of functions $f$ such that $\frac{d}{d\sigma}f \in L^2(\partial \Omega)$. Since the embedding $H^1(\partial \Omega) \hookrightarrow H^{1/2}(\partial \Omega)$ is compact this is sufficient.

 For $f \in H^{1/2}(\partial \Omega)$, let $Df \in H^1(\Omega)$ denote the double-layer potential of $f$ in the interior of $\Omega$. The tangential derivative is associated with the classical jump formula \cite{Maue}
\[\frac{d}{d\sigma} \Tr Df = \frac{1}{2}\left(\frac{d}{d\sigma}f-K^*\frac{d}{d\sigma}f\right). \]
The validity of this formula for all $f \in H^{1/2}(\partial \Omega)$ follows from the classical considerations by approximation and the continuity of all operators involved. Combined with the jump formula \eqref{eq:jumpformula} we conclude that 
\begin{align*}
\frac{d}{d\sigma} Kf(x) &= -K^*\frac{d}{d\sigma}f(x) = -\frac{2}{\pi} \pv\int_{\partial \Omega}  \partial_{n_x} \log|x-y| \frac{d}{d\sigma}f(y) \, d\sigma(y) \\
&= \frac{2}{\pi} \pv\int_{\partial \Omega}  \frac{d}{d\sigma(y)}\partial_{n_x} \log|x-y| (f(y)-f(x)) \, d\sigma(y),
\end{align*}
where the last step equality follows from integration by parts.
Note that if $x, y \in \supp \rho_{N+1}$, then
\[\left|\frac{d}{d\sigma(y)}\partial_{n_x} \log|x-y|\right| \lesssim \frac{1}{|x-y|}.\]
If instead $x \in \supp \rho_j$ and $y \in \supp \rho_k$, $k \neq j$, then $\frac{d}{d\sigma(y)}\partial_{n_x} \log|x-y|$ is bounded. Hence,
\begin{multline} \label{eq:cpctineq}
\left| \frac{d}{d\sigma} M_{\rho_j}KM_{\rho_k}f(x) - \left[\frac{d}{d\sigma}\rho_j(x)\right] K(\rho_k f)(x) \right| \\ \lesssim  |\rho_j(x)| \int_{\partial \Omega} \frac{|\rho_k(x)f(x)- \rho_k(y)f(y)|}{|x-y|} \, d \sigma(y) 
\end{multline}
 It is clear that the term $\left[\frac{d}{d\sigma}\rho_j(x)\right] K(\rho_k f)(x)$ is in $L^2(\partial \Omega)$, by Lemma \ref{lem:multbound} and the boundedness of $K$ on $H^{1/2}(\partial \Omega)$. It remains to show that the right-hand side of \eqref{eq:cpctineq} is in $L^2(\partial \Omega)$. However, this follows from the Cauchy-Schwarz inequality and Lemma \ref{lem:multbound}:
 \begin{multline*}
 \int_{\partial \Omega} \left| \rho_j(x) \int_{\partial \Omega} \frac{|\rho_k(x)f(x)- \rho_k(y)f(y)|}{|x-y|} \, d\sigma(y) \right|^2 \, d\sigma(x) \\ \lesssim \int_{\partial \Omega} \int_{\partial \Omega} \frac{|\rho_k(x)f(x)- \rho_k(y)f(y)|^2}{|x-y|^2} \, d\sigma(y) \, d\sigma(x) \\ \lesssim \|M_{\rho_k} f\|^2_{H^{1/2}(\partial \Omega)} \lesssim \| f\|^2_{H^{1/2}(\partial \Omega)}. \qedhere
 \end{multline*}
\end{proof}

\section{The Wedge} \label{sec:wedge}
Let $W_\alpha = \{z \in \mathbb{C} \, : \, |\Arg z | < \alpha/2 \}$ be a wedge of aperture $\alpha$, $0 < \alpha < 2\pi$. Consider any linear fractional transformation $L$ which maps $W_{\alpha} \cup \{\infty\}$ onto a bounded domain. Since $T_\Omega$ is unitarily equivalent to $T_{L(\Omega)}$ for any domain $\Omega$ \cite{PP14}, it follows that the spectrum of the Neumann--Poincar\'e operator $K$ associated with $L(W_\alpha)$ may be determined by considering $T_{W_\alpha}$.\footnote{We avoid considering $K$ directly on $\partial W_\alpha$ due to the technical difficulties associated with unbounded domains.} This section is devoted to proving the following.

\begin{thm} \label{thm:specwedge}
The spectrum of $K \colon H^{1/2} \left( \partial L(W_\alpha) \right) \to H^{1/2} \left( \partial L(W_\alpha) \right)$ is given by
\[
\sigma(K) = \left\{x\in \mathbb{R} \, : \, |x| \leq \left|1-\frac{\alpha}{\pi}\right| \right\} \cup \{1\}.
\]
$1$ is a a simple eigenvalue. The remainder of the spectrum is essential, of uniform multiplicity 2.
\end{thm}
\begin{rem}
This result, stated somewhat differently, also appears in the preprint \cite{KLY15}. Our proof is rather different
and we include it with full details below.
\end{rem}
It's sufficient to consider $\alpha < \pi$, because the spectrum of $T_{W_\alpha}$ is the same as that of $T_{W_{2\pi - \alpha}}$, since $W_{2\pi-\alpha}$ is a rotation of $\overline{W_\alpha}^c$. We begin with the following simple proposition about the kernels of $L^2(W_\alpha)$.

\begin{prop}[\cite{PS00}]
$L^2(W_\alpha)$ is a reproducing kernel Hilbert space. The reproducing kernel at the point $z \in W_\alpha$ is given by
\begin{equation} \label{eq:repkernel}
k_z(w) =  \frac{1}{\alpha^2} \frac{w^{\pi/\alpha-1}\overline{z}^{\pi/\alpha-1}}{(w^{\pi/\alpha}+\overline{z}^{\pi/\alpha})^2}.
\end{equation}
\end{prop}

Now let $R_{\pm} = \{z \in \mathbb{C} \, : \, |\Arg z| = \pm \alpha/2\}$ be the two rays of the boundary $\partial W_\alpha$. We consider the holomorphic Schwarz functions $S_\pm$ on $W_\alpha$ such that $S_\pm(\zeta) = \bar{\zeta}$ on $R_\pm$,
\[
S_\pm(\zeta) = e^{\mp i \alpha} \zeta.
\]
By first applying Stokes' theorem and then Cauchy-Goursat's theorem to each of the rays $R_+$ and $R_-$, we find for functions $f$ with sufficient decay that for $z \in W_\alpha$
\begin{align*}
\overline{T_{W_\alpha} f(z)} &= \lim_{\varepsilon \to 0} \frac{i}{2\pi} \left[ \int_{\partial W_\alpha} \frac{f(\zeta)}{\bar{\zeta}-\bar{z}} \, d\zeta - \frac{1}{\varepsilon^2} \int_{|\zeta-z|=\varepsilon} f(\zeta)(\zeta-z) \, d\zeta \right] \\ &= \frac{i}{2\pi} \left[ \int_{R_-} \frac{f(\zeta)}{S_-(\zeta)-\bar{z}} \, d\zeta - \int_{R_+} \frac{f(\zeta)}{S_+(\zeta)-\bar{z}} \, d\zeta    \right] \\
&= \frac{i}{2\pi} \int_0^\infty f(y) \left[\frac{1}{e^{i \alpha}y -\bar{z}} - \frac{1}{e^{- i \alpha}y -\bar{z}} \right] \, dy \\
&= \frac{\sin \alpha}{\pi} \int_0^\infty f(y) \frac{y}{(e^{i \alpha}y -\bar{z})(e^{- i \alpha}y -\bar{z})} \, dy.
\end{align*}
For instance, this formula is valid for linear combinations of kernels \eqref{eq:repkernel}. Hence, it follows for $x > 0$ that
\begin{align}  \label{eq:twcomp}
\begin{split}
T_{W_\alpha} f(x) &= \frac{\sin \alpha}{\pi} \int_0^\infty \overline{f(y)} \frac{y}{|e^{i \alpha}y -x|^2} \, dy\\  &= \frac{\sin \alpha}{\pi} \int_{-\infty}^\infty \overline{e^{t}f(e^{t}x)} \frac{e^{t}}{|e^{i\alpha}e^t - 1|^2}\, dt,
\end{split}
\end{align}
motivating the following lemma.
\begin{lem} \label{lem:Acomp}
For $t \in \mathbb{R}$, let $U_t : L^2_a(W_\alpha) \to L^2_a(W_\alpha)$ be the operator
\[U_tf(z) = e^tf(e^tz).\]
Then $\left(U_t\right)_{t\in\mathbb{R}}$ is a strongly continuous group of unitary operators with generator $iA$, $U_t = e^{itA}$, where 
\[Af(z) =  -i(f(z) + zf'(z)).\]
$A$ is a (densely defined) self-adjoint operator with full spectrum, $\sigma(A) = \mathbb{R}$. Furthermore, its spectrum has uniform multiplicity $1$.
\end{lem}
\begin{proof}
The verification that $U_t$ is a strongly continuous group of unitaries is straightforward. The formula for $A$ follows immediately from the fact that $iA$ is the strong limit of $t^{-1}(U_t - I)$ as $t\to 0$. For $t$ real, $z^{it}$ is bounded from below and above in $W_\alpha$. Hence the operator $M_t$ of multiplication by $z^{it}$ is bounded and invertible on $L^2_a(W_\alpha)$. A computation shows that
\[M_t^{-1}AM_t = A + tI.\]
Hence $A$ is similar to $A + t$ for every $t\in\mathbb{R}$. Since the spectrum of $A$ is not empty, it must therefore be full. 

To show that the spectrum of $A$ has multiplicity $1$, we prove that for every $z \in W_\alpha$, the reproducing kernel $k_z$ of $L^2(W_\alpha)$ at $z$, see \eqref{eq:repkernel}, is a cyclic vector of $A$. Suppose that $f \in L^2(W_\alpha)$ is orthogonal to $\textrm{span} \{A^n k_z \, : \, n\geq 0 \}$. Clearly every function $g_n(w) = w^n k_z^{(n)}(w)$ is in this linear span. Note that 
\[g_n(w) = \left(\frac{d}{d\rho}\right)^n k_z(\rho w)\Big|_{\rho=1},\]
and that for $\rho > 0$
\[k_z(\rho w) = \frac{1}{\rho^2} k_{z/\rho} (w).\]
Hence we have that
\[0 = \langle f, g_n \rangle_{L^2(W_\alpha)} =  \left(\frac{d}{d\rho}\right)^n \rho^{-2} f(z/\rho) \Big|_{\rho=1}.\]
Evaluating for $n=0, 1, 2, \ldots$ we find that $f^{(n)}(z) = 0$ for every $n \geq 0$, proving that $f = 0$.
\end{proof}

Let $J : L^2_a(W_\alpha) \to L^2_a(W_\alpha)$ be the (anti-linear) conjugation given by $Jf(z) = \overline{f(\bar{z})}$. Then $T_{W_\alpha} J$ is a self-adjoint operator on $L^2_a(W_\alpha)$ which additionally is $J$-symmetric \cite{GPP14} in the sense that $T_{W_\alpha} J = J(T_{W_\alpha} J) J = J T_{W_\alpha}$. 

\begin{lem} \label{lem:specfa}
For $0 < \alpha < \pi$, $T_{W_\alpha} J$ is a positive operator, with spectrum 
\[\sigma(T_{W_\alpha} J) = \left \{x \in \mathbb{R} \, : \, 0 \leq x \leq 1-\frac{\alpha}{\pi} \right\} \]
of uniform multiplicity $2$.
\end{lem}
\begin{proof}
For $z$ with $|\Im z| < 1$, let
\[F(z) = \frac{\sin\alpha}{\pi}\int_{-\infty}^{\infty}e^{itz} \frac{e^t}{|e^{i\alpha}e^t-1|^2}\, dt.  \]
Then $F(A)$ is the operator on $L^2_a(W_\alpha)$ such that
\begin{align*}
F(A) f(z) &= \frac{\sin\alpha}{\pi}\int_{-\infty}^{\infty}e^{itA} f(z) \frac{e^t}{|e^{i\alpha}e^t-1|^2}\, dt \\ &= \frac{\sin\alpha}{\pi}\int_{-\infty}^{\infty} e^tf(e^tz) \frac{e^t}{|e^{i\alpha}e^t-1|^2}\, dt. 
\end{align*}
Comparing with \eqref{eq:twcomp} it is now clear that $F(A) = T_{W_\alpha} J$. 

We already know that $F(A) = T_{W_\alpha} J$ is a self-adjoint operator. The change of variable $s=e^{it}$ gives us that
\[F(x) = \int_0^\infty s^{ix}\frac{s}{s^2 - 2s \cos\alpha +1} \, \frac{ds}{s}, \qquad x \in \mathbb{R}.\]
This Mellin transform can be computed by ``partial fractions'' (see also \cite{Mitrea02}, pp. 453), which yields
\[F(x) = \frac{\sin(i(\pi-\alpha)x)}{\sin(i\pi x)}, \qquad x \in \mathbb{R}.\]
$F$ is thus a smooth even positive function on $\mathbb{R}$, such that $F(0) = 1-\alpha/\pi$, $F$ is decreasing for $x \geq 0$ and $F(x)\to 0$ as $x \to \infty$. The statement of the lemma now follows in view of Lemma~\ref{lem:Acomp} and the spectral theorem for unbounded self-adjoint operators with its associated multiplicity theory. See \cite{Nelson69} for a remarkably clear presentation of the multiplicity theory, and \cite{AK72} for its application to the pushforward measure $\mu \circ F^{-1}$, $\mu$ a scalar spectral measure for $A$.
\end{proof}
We can now prove the theorem by a comparison of $T_{W_\alpha}$ and $T_{W_\alpha} J$.

\begin{proof}[Proof of Theorem \ref{thm:specwedge}.]
Since $T_{W_\alpha}J$ is self-adjoint and $J$-symmetric, it holds that $JT_{W_\alpha} J = T_{W_\alpha}$. Therefore $(T_{W_\alpha}J)^2 = (T_{W_\alpha})^2$. Hence the theorem follows from Lemma \ref{lem:specfa}, the spectral theorem, and the symmetry of the spectrum of $T_{W_\alpha}$.
\end{proof}

%$\sigma_\mathbb{R}(T_{W_\alpha})$ is symmetric about the origin, so it is enough to consider $\lambda > 0$. For the sake of clarity, let us first consider the case of an eigenvalue $\lambda$.
%Suppose that $f \neq 0$ satisfies $T_{W_\alpha} f = \lambda f$. Then we see from \eqref{eq:twcomp} that $T_{W_\alpha} J f(z) = \lambda J f(z)$, first for real $z=x$, then for all $z$ by analytic extension. It could not be that $f = -Jf$, for then $T_{W_\alpha} Jf = -\lambda f$, a contradiction to Lemma \ref{lem:specfa}. Hence $f+Jf$ is a non-zero eigenvector for $T_{W_\alpha} J$, since $J(f+Jf) = f+Jf$ and $T_{W_\alpha}(f + Jf) = \lambda(f+Jf)$.

%Conversely, suppose that $T_{W_\alpha} Jf = \lambda f$. Then $T_{W_\alpha} f = \lambda Jf$ so that $f + Jf$ is an eigenvector of $T_{W_\alpha}$. If $f = -Jf$, then $if$ is an eigenvector.

%We know from the general theory of \cite{PP14} that any point of $\sigma_\mathbb{R}(T_{W_\alpha})$ is an approximate eigenvalue. As for $T_{W_\alpha} J = F(A)$, its adjoint is given by 
%\[F(A)^\ast f(z) = \frac{\sin\alpha}{\pi}\int_{-\infty}^{\infty}e^{-itA} f(z) \frac{e^t}{|e^{i\alpha}e^t-1|^2}\, dt,\]
%from which it is plain that $T_{W_\alpha} J$ is a normal operator. Hence, by the spectral theorem, its spectrum is also characterized by approximate eigenvalues. The above argument for eigenvectors shows almost verbatim that $\lambda \in \mathbb{R}$ is an approximate eigenvalue for $T_{W_\alpha}$ if and only if it is an approximate eigenvalue for $T_{W_\alpha} J$.

\section{General curvilinear polygons}
In this section we shall completely determine the essential spectrum of the Neumann--Poincar\'e operator associated with a curvilinear polygon $\Omega \subset \mathbb{C}$.
\begin{thm} \label{thm:specomega}
Let $K : H^{1/2}(\partial \Omega) \to H^{1/2}(\partial \Omega)$ be the Neumann--Poincar\'e operator of a $C^{2}$-smooth curvilinear polygon $\Omega \subset \mathbb{C}$ with angles $\alpha_1, \ldots, \alpha_N$. Then
\[
\sigma_{\ess}(K) = \left\{x\in \mathbb{R} \, : \, |x| \leq \max_{1 \leq j \leq N}\left|1-\frac{\alpha_j}{\pi}\right| \right\}.
\]
\end{thm}
For an operator $T : H^{1/2}(\partial \Omega) \to H^{1/2}(\partial \Omega)$, denote by $\sigma_{\ea}(T)$ the essential spectrum of $T$ in the sense of approximate eigenvalues. That is, $\lambda \in \sigma_{\ea}(T)$ if and only if there is a bounded sequence $(f_n)_{n=1}^\infty \subset H^{1/2}(\partial \Omega)$ having no convergent subsequence, such that $(T-\lambda)f_n \to 0$. We call $(f_n)$ a singular sequence. Note that if  $S : H^{1/2}(\partial \Omega) \to H^{1/2}(\partial \Omega)$ is another operator such that $S-T$ is compact, then $\sigma_{\ea}(S) = \sigma_{\ea}(T)$.
\begin{lem}
For the Neumann--Poincar\'e operator $K : H^{1/2}(\partial \Omega) \to H^{1/2}(\partial \Omega)$ it holds that
\[\sigma_{\ess}(K) = \sigma_{\ea}(K).\]
\end{lem}
\begin{proof}
Consider first $K$ acting on $H^{1/2}(\partial \Omega)/\mathbb{C}$, which only eliminates the simple isolated eigenvalue $\lambda = 1$ from the spectrum of $K$ ($K1 = 1$). Since we also know that $K$ is similar to a self-adjoint operator on $H^{1/2}(\partial \Omega)/\mathbb{C}$, we have that
\[\sigma_{\ess}(K) = \sigma_{\ess}\left(K|_{H^{1/2}(\partial \Omega)/\mathbb{C}}\right) = \sigma_{\ea}\left(K|_{H^{1/2}(\partial \Omega)/\mathbb{C}}\right) \subset \sigma_{\ea}(K).\]
The reverse inclusion is obvious.
\end{proof}
We now begin the proof of Theorem \ref{thm:specomega} with a localization lemma.  Recall that for a smooth function $\rho$ on $\partial \Omega$, we denote by $M_\rho$ the operator of multiplication by $\rho$ on $H^{1/2}(\partial \Omega)$. 

\begin{lem} \label{lem:omegalocal}
Let $\Omega$ be a $C^{2}$-smooth curvilinear polygon with vertices $(a_j)_{j=1}^N$. For each $j$, $1 \leq j\leq N$, let $\rho_j$ be a smooth function on $\partial \Omega$ such that $\rho(x) = 1$ for all $x$ in a neighborhood of $a_j$. Furthermore, suppose that the supports of $\rho_j$ are pairwise disjoint (at a positive distance apart). Then, letting $K : H^{1/2}(\partial \Omega) \to H^{1/2}(\partial \Omega)$ be the Neumann--Poincar\'e operator,
\begin{equation} \label{eq:local}
K - \sum_{j=1}^N M_{\rho_j} K M_{\rho_j} \textrm{ is compact.}
\end{equation}
Furthermore, 
\begin{equation} \label{eq:specunion}
\sigma_{\ess}(K) = \bigcup_{j=1}^N \sigma_{\ea}(M_{\rho_j} K M_{\rho_j}).
\end{equation}
\end{lem}
\begin{proof}
We construct a smooth partition of unity of $\partial \Omega$ by letting $\rho_{N+1} = 1 - \sum_{j=1}^N \rho_j$.
We then have that
\[K = \sum_{j=1}^{N+1} \sum_{k=1}^{N+1} M_{\rho_j}KM_{\rho_k}. \]
If $j\neq k$ or $j=k=N+1$, then, by Lemma \ref{lem:cutoffcpct}, $M_{\rho_j}KM_{\rho_k}$ is compact. This gives us the validity of \eqref{eq:local}, and hence that
\[\sigma_{\ess}(K) = \sigma_{\ea}(K) = \sigma_{\ea}\left(\sum_{j=1}^N M_{\rho_j} K M_{\rho_j}\right).\]

$\lambda = 0$ clearly belongs to both sides of \eqref{eq:specunion}. Suppose now that $0 \neq \lambda \in \sigma_{\ea}(M_{\rho_k}KM_{\rho_k})$ for some $k$. Let $(f_n)_n$ be a corresponding singular sequence, so that $(M_{\rho_k}KM_{\rho_k} - \lambda)f_n \to 0$ as $n \to \infty$. Multiplying on the left by $M_{\rho_j}$ it follows that $M_{\rho_j} f_n \to 0$ for every $j\neq k$, so that $(f_n)$ is a singular sequence also for $\sum_j M_{\rho_j}KM_{\rho_j}$, proving that
\[ \bigcup_{j=1}^N\sigma_{\ea}(M_{\rho_j}KM_{\rho_j}) \subset \sigma_{\ea} \left( \sum_{j=1}^N M_{\rho_j}KM_{\rho_j} \right). \]

Conversely, suppose that 
\begin{equation}\label{eq:seplambda}
\sum_{j=1}^N M_{\rho_j}KM_{\rho_j}f_n - \lambda f_n \to 0
\end{equation}
 as $n\to\infty$, for a sequence $(f_n)$ with no convergent subsequence. Let $\chi$ be a smooth function which is $1$ on the support of $\rho_1$, $0$ on the support of $\rho_j$ for every $j \neq 1$. Multiplying \eqref{eq:seplambda} with $\chi$ and noting that $M_\chi M_{\rho_1} = M_{\rho_1}  M_\chi = M_{\rho_1}$ and that $M_\chi M_{\rho_j} = 0$ for every other $j$, it follows that $M_{\rho_1}KM_{\rho_1} M_\chi f_n - \lambda M_\chi f_n \to 0$. Hence  $M_\chi f_n$ is a singular sequence for $M_{\rho_1}KM_{\rho_1}$, unless $(M_\chi f_n)$ has a convergent subsequence $(M_\chi f_k)$.  In the latter case $M_{1-\chi} f_k$ is a singular sequence for $\sum_{j=2}^N M_{\rho_j}KM_{\rho_j}$. Now the argument of this paragraph may be repeated until one finds a singular sequence for $ M_{\rho_j}KM_{\rho_j}$, for some $j$. We have hence proved that
\[ \sigma_{\ea}\left( \sum_{j=1}^N M_{\rho_j}KM_{\rho_j} \right) \subset \bigcup_{j=1}^N\sigma_{\ea}(M_{\rho_j}KM_{\rho_j}). \qedhere\]
\end{proof}
Let $L(z) = (z+1)/(z-1)$, and let $V_\alpha = L(W_\alpha)$, where $W_\alpha$ is the wedge of the preceding section, $0 < \alpha < 2\pi$. Then $V_\alpha$ is a lens domain, symmetric around the real and imaginary axis, with corners of angle $\alpha$ at $-1$ and $1$. The next lemma says that the two corners have equal contribution to the essential spectrum of $K : H^{1/2}(\partial V_\alpha) \to H^{1/2}(\partial V_\alpha)$. 

\begin{lem}\label{lem:wedgelocal}
Let $\rho$ be a smooth function on $\partial V_\alpha$, compactly supported in the left half-plane, such that $\rho(x) = 1$ for all $x$ in a neighborhood of $-1$. Then 
\[\sigma_{\ea}(M_{\rho}KM_{\rho}) = \sigma_{\ess}(K) = \left\{x\in \mathbb{R} \, : \, |x| \leq \left|1-\frac{\alpha}{\pi}\right| \right\}.\]
\end{lem}
\begin{proof}
Let $\rho_2$ be the function obtained by reflecting $\rho_1$ in the imaginary axis. Then, by symmetry, $M_{\rho_1}KM_{\rho_1}$ is unitarily equivalent to $M_{\rho_2}KM_{\rho_2}$, and hence the two operators have the same spectrum. Applying Lemma \ref{lem:omegalocal} and Theorem \ref{thm:specwedge} we obtain the desired conclusion.
\end{proof}
Using results on perturbations by conformal mappings from \cite{PP14} allows us to handle a general corner of opening $\alpha$, not only the one coming from the wedge.
\begin{lem} \label{lem:onecorner}
Let $\Omega$ be a $C^{2}$-smooth curvilinear polygon, and let one of its vertex points be $a_j$, with corresponding angle $\alpha_j$.
Let $\rho$ be a smooth function on $\partial \Omega$ such that $\rho(x) = 1$ for all $x$ in a neighborhood of $a_j$. Then, if the support of $\rho$ is sufficiently small, it holds that
\begin{equation} \label{eq:speclocal}
\sigma_{\ea}(M_{\rho}KM_{\rho}) = \left\{x\in \mathbb{R} \, : \, |x| \leq \left|1-\frac{\alpha_j}{\pi}\right| \right\},
\end{equation}
where $K : H^{1/2}(\partial \Omega) \to H^{1/2}(\partial \Omega)$ is the Neumann--Poincar\'e operator of $\Omega$.
\end{lem} 
\begin{proof}
Due to the local nature of the operator $M_{\rho}KM_{\rho}$ we may clearly assume, without loss of generality, that $\Omega$ only has a single corner $a$, of angle $\alpha$.  Similarly, let $U_{\alpha}$ be a smooth domain with only one corner. We suppose that this corner is at $-1$, and that $U_\alpha$ is identical to $V_\alpha$ in a neighborhood of $-1$. Lemma \ref{lem:wedgelocal} then produces a function $\chi$ on $\partial U_\alpha$ such that
\[\sigma_{\ea}(M_{\chi}K_{U_\alpha}M_{\chi}) = \sigma_{\ess}(K_{V_{\alpha}}) = \left\{x\in \mathbb{R} \, : \, |x| \leq \left|1-\frac{\alpha}{\pi}\right| \right\}.\]
On the other hand, the difference $K_{U_\alpha} - M_{\chi}K_{U_\alpha}M_{\chi}$ is compact in this case, since there is only one corner. Hence,
\[\sigma_{\ess}(K_{U_\alpha}) = \sigma_{\ea}(K_{U_\alpha}) = \left\{x\in \mathbb{R} \, : \, |x| \leq \left|1-\frac{\alpha}{\pi}\right| \right\}.\]
Let $\varphi : \Omega \to U_\alpha$ be a Riemann map such that $\varphi(a) = -1$. Lemma 4.3 of \cite{PP14} then shows that $\varphi$ is $C^{1,b}$-smooth in $\overline{\Omega}$, for $0 < b < 1$, and \cite[Lemma 4.4]{PP14} then says that $\sigma_{\ess}(K) = \sigma_{\ess}(K_{U_\alpha})$. The proof is finished by noting that $\sigma_{\ess}(K) = \sigma_{\ea}(M_\rho K M_\rho)$ by Lemma \ref{lem:omegalocal} applied to $\Omega$.
\end{proof}
\begin{proof}[Proof of Theorem \ref{thm:specomega}] The statement follows immediately from Lemmas \ref{lem:omegalocal}, \ref{lem:wedgelocal} and \ref{lem:onecorner}.
\end{proof}

\subsection{Final Remarks} Since $K : H^{1/2}(\partial \Omega)/\mathbb{C} \to H^{1/2}(\partial \Omega)/\mathbb{C}$ is similar to a self-adjoint operator, it is completely described by a spectral resolution. In other words, it is a Dunford scalar operator \cite{DS71}. In the absence of a multiplicity theory (and hence classification) of Dunford scalar operators modulo compact operators, we gather a few observations which might lead to a better framework to explain the phenomena unveiled by the computations specific to the NP operator. We keep the notation of the previous sections, but keep in a mind a more general situation. 

Let $K$ denote the Neumann--Poincar\'e operator acting in the complex Hilbert space $H = H^{1/2}(\partial \Omega)/\mathbb{C}$ and let $K_j$ denote its localizations (in our case $K_j = M_{\rho_j} K M_{\rho_j}$). We can assume that the supports $F_j$ of the cut-off functions $\rho_j$ are separated, so that the operator $K_j$ acts
on the closed subspace $H_j$ of elements of $H$ having support contained in $F_j$, for every $j$. Let $P_j$ denote the orthogonal projection of $H$ onto $H_j$.

The subspaces $H_j$ are not mutually orthogonal due to the non-locality of $H$, but their operator angles are almost perpendicular in the sense that $P_j P_k$ is compact for every $ j \neq k$. Denote by $\widetilde{T}$ the class of an operator
in the Calkin algebra ${\mathcal L}(H)/{\mathcal K}(H)$. Thus $\widetilde{P_j}$ are mutually orthogonal projections and
$$ \widetilde{K_j} = \widetilde{P_j}\widetilde{K}\widetilde{P_j}.$$
Moreover,
$$ \widetilde{K} = \widetilde{K_1} + \widetilde{K_2} + \ldots \widetilde{K_N}.$$
From here we infer that for every polynomial $q \in \mathbb{C} [z]$ we have
$$  q(\widetilde{K_j}) = \widetilde{P_j}q(\widetilde{K})\widetilde{P_j}.$$

As $K$ itself is a Dunford scalar operator with real spectrum, contained in the compact set $\sigma \subset \mathbb{R}$, we infer
$$ \| q(\widetilde{K_j}) \| \leq \| q(\widetilde{K}) \| \leq C \| q \|_{\infty, \sigma},$$
where $C$ is a constant. By the Stone-Weierstrass theorem, we find that every component $\widetilde{K_j}$ admits a continuous functional calculus with
continuous functions on $\sigma$. In this sense, every $\widetilde{K_j}$ is a scalar operator in the Calkin algebra, with real spectrum. Their direct orthogonal sum
is the class $\widetilde{K}$ of the Neumann--Poincar\'e operator, and in this manner the essential spectra of the components $\widetilde{K_j}$ stack on top of each other.

\section*{Acknowledgments}
The first author is grateful to Alexandru Aleman for discussions that proved essential to writing this article. The second author is grateful to Hyeonbae Kang
and Habib Ammari for their unconditional interest in the
 spectral analysis of the NP operator.

\bibliographystyle{amsplain} 
\bibliography{NPSpectrum} 

\providecommand{\bysame}{\leavevmode\hbox to3em{\hrulefill}\thinspace}
\providecommand{\MR}{\relax\ifhmode\unskip\space\fi MR }
\providecommand{\MRhref}[2]{%
  \href{http://www.ams.org/mathscinet-getitem?mr=#1}{#2}
}
\providecommand{\href}[2]{#2}
\begin{thebibliography}{10}

\bibitem{AK72}
M.~B. Abrahamse and Thomas~L. Kriete, \emph{The spectral multiplicity of a
  multiplication operator}, Indiana Univ. Math. J. \textbf{22} (1972/73),
  845--857.

\bibitem{Ahlfors52}
Lars~V. Ahlfors, \emph{Remarks on the {N}eumann-{P}oincar\'e integral
  equation}, Pacific J. Math. \textbf{2} (1952), 271--280.

\bibitem{ACKLM13}
Habib Ammari, Giulio Ciraolo, Hyeonbae Kang, Hyundae Lee, and Graeme~W. Milton,
  \emph{Spectral theory of a {N}eumann-{P}oincar\'e-type operator and analysis
  of cloaking due to anomalous localized resonance}, Arch. Ration. Mech. Anal.
  \textbf{208} (2013), no.~2, 667--692.

\bibitem{ACKLY13}
Habib Ammari, Giulio Ciraolo, Hyeonbae Kang, Hyundae Lee, and Kihyun Yun,
  \emph{Spectral analysis of the {N}eumann-{P}oincar\'e operator and
  characterization of the stress concentration in anti-plane elasticity}, Arch.
  Ration. Mech. Anal. \textbf{208} (2013), no.~1, 275--304.

\bibitem{ADKL14}
Habib Ammari, Youjun Deng, Hyeonbae Kang, and Hyundae Lee, \emph{Reconstruction
  of inhomogeneous conductivities via the concept of generalized polarization
  tensors}, Ann. Inst. H. Poincar\'e Anal. Non Lin\'eaire \textbf{31} (2014),
  no.~5, 877--897. \MR{3258359}

\bibitem{AK04}
Habib Ammari and Hyeonbae Kang, \emph{Reconstruction of small inhomogeneities
  from boundary measurements}, Lecture Notes in Mathematics, vol. 1846,
  Springer-Verlag, Berlin, 2004.

\bibitem{BS51}
S.~Bergman and M.~Schiffer, \emph{Kernel functions and conformal mapping},
  Compositio Math. \textbf{8} (1951), 205--249.

\bibitem{Carleman16}
Torsten Carleman, \emph{{\"U}ber das {N}eumann-{P}oincar\'esche problem f\"ur
  ein gebiet mit ecken}, Almquist \& Wiksells, Uppsala, 1916.

\bibitem{DS71}
Nelson Dunford and Jacob~T. Schwartz, \emph{Linear operators. {P}art {III}:
  {S}pectral operators}, Interscience Publishers [John Wiley \& Sons, Inc.],
  New York-London-Sydney, 1971, With the assistance of William G. Bade and
  Robert G. Bartle, Pure and Applied Mathematics, Vol. VII.

\bibitem{GPP14}
Stephan~Ramon Garcia, Emil Prodan, and Mihai Putinar, \emph{Mathematical and
  physical aspects of complex symmetric operators}, J. Phys. A \textbf{47}
  (2014), no.~35, 353001, 54.

\bibitem{HKL16}
Johan Helsing, Hyeonbae Kang, and Mikyoung Lim, \emph{Classification of
  spectrum of the neumann-poincar\'e operator on planar domains with corners by
  resonance: {A} numerical study}, preprint (2016).

\bibitem{KLY15}
Hyeonbae Kang, Mikyoung Lim, and Sanghyeon Yu, \emph{Spectral resolution of the
  {N}eumann-{P}oincar\'e operator on intersecting disks and analysis of plasmon
  resonance}, arXiv:1501.02952 [math.AP] (2015).

\bibitem{Kress}
Rainer Kress, \emph{Linear integral equations}, third ed., Applied Mathematical
  Sciences, vol.~82, Springer, New York, 2014.

\bibitem{Krushkal09}
Samuel Krushkal, \emph{Fredholm eigenvalues of {J}ordan curves: geometric,
  variational and computational aspects}, Analysis and mathematical physics,
  Trends Math., Birkh\"auser, Basel, 2009, pp.~349--368.

\bibitem{Maue}
A.-W. Maue, \emph{Zur {F}ormulierung eines allgemeinen {B}eugungsproblems durch
  eine {I}ntegralgleichung}, Z. Physik \textbf{126} (1949), 601--618.
  \MR{0032455 (11,293i)}

\bibitem{Mitrea02}
Irina Mitrea, \emph{On the spectra of elastostatic and hydrostatic layer
  potentials on curvilinear polygons}, J. Fourier Anal. Appl. \textbf{8}
  (2002), no.~5, 443--487.

\bibitem{MS15}
Yoshihisa Miyanishi and Takashi Suzuki, \emph{Eigenvalues and eigenfunctions of
  double layer potentials}, arXiv:1501.03627 [math.SP] (2015).

\bibitem{Nelson69}
Edward Nelson, \emph{Topics in dynamics. {I}: {F}lows}, Mathematical Notes,
  Princeton University Press, Princeton, N.J.; University of Tokyo Press,
  Tokyo, 1969.

\bibitem{PP14}
Karl-Mikael Perfekt and Mihai Putinar, \emph{Spectral bounds for the
  {N}eumann-{P}oincar\'e operator on planar domains with corners}, J. Anal.
  Math. \textbf{124} (2014), 39--57.

\bibitem{PS00}
Mihai Putinar and Harold~S. Shapiro, \emph{The {F}riedrichs operator of a
  planar domain}, Complex analysis, operators, and related topics, Oper. Theory
  Adv. Appl., vol. 113, Birkh\"auser, Basel, 2000, pp.~303--330. \MR{1771771
  (2001g:47049)}

\bibitem{Radon}
Johann Radon, \emph{Gesammelte {A}bhandlungen. {B}and 1}, Verlag der
  \"Osterreichischen Akademie der Wissenschaften, Vienna; Birkh\"auser Verlag,
  Basel, 1987, With a foreword by Otto Hittmair, Edited and with a preface by
  Peter Manfred Gruber, Edmund Hlawka, Wilfried N{\"o}bauer and Leopold
  Schmetterer.

\bibitem{TW09}
Marius Tucsnak and George Weiss, \emph{Observation and control for operator
  semigroups}, Birkh\"auser Advanced Texts: Basler Lehrb\"ucher. [Birkh\"auser
  Advanced Texts: Basel Textbooks], Birkh\"auser Verlag, Basel, 2009.

\bibitem{Wendland09}
W.~L. Wendland, \emph{On the double layer potential}, Analysis, partial
  differential equations and applications, Oper. Theory Adv. Appl., vol. 193,
  Birkh\"auser Verlag, Basel, 2009, pp.~319--334.

\bibitem{Werner97}
Stephan Werner, \emph{Spiegelungskoeffizient und {F}redholmscher {E}igenwert
  f\"ur gewisse {P}olygone}, Ann. Acad. Sci. Fenn. Math. \textbf{22} (1997),
  no.~1, 165--186.

\bibitem{Zaremba04}
S.~Zaremba, \emph{Les fonctions fondamentales de {M}. {P}oincar\'e et la
  m\'ethode de {N}eumann pour une fronti\`ere compos\'ee de polygones
  curvilignes}, Journal de Math\'ematiques Pures et Appliqu\'ees \textbf{10}
  (1904), 395--444.

\end{thebibliography}
\end{document}